\documentclass[12pt]{article}
\pdfoutput=1

\usepackage[utf8]{inputenc}
\usepackage{geometry}
\usepackage{amsmath,amsfonts,amsthm,amscd,upref,amstext, amssymb}
\usepackage{verbatim}
\usepackage[shortlabels]{enumitem}
\usepackage{xcolor}

\usepackage[
backend=biber,
style=alphabetic,
]{biblatex}
\title{A bibLaTeX example}
\usepackage{url}
\newtheorem{theorem}{Theorem}[section]
\newtheorem{corollary}[theorem]{Corollary}
\newtheorem{lemma}[theorem]{Lemma}
\newtheorem{claim}[theorem]{Claim}
\newtheorem{subclaim}[theorem]{Subclaim}

\newtheorem{definition}[theorem]{Definition}
\newtheorem{conjecture}[theorem]{Conjecture}
\newtheorem{remark}[theorem]{Remark}

\renewcommand\H{\mathcal{H}}

\newcommand\R{\mathbb{R}}

\usepackage{fancyhdr}
\usepackage{lipsum} 

\fancypagestyle{plain}{
  \fancyhf{}
  \fancyfoot[C]{\thepage}%
}
\fancypagestyle{firstpage}{
  \fancyhf{}
  \fancyfoot[L]{}%
}
\title{More Derived Models in PFA}
\author{Derek Levinson,\footnote{Department of Mathematics, University of North Texas, Denton, TX USA. Email: Derek.Levinson@unt.edu.}\, Nam Trang,\footnote{Department of Mathematics, University of North Texas, Denton, TX USA. Email: Nam.Trang@unt.edu. Nam Trang is partially supported by the NSF via CAREER grant DMS-1945592 and the Simons Foundation via the Simons Fellowship.}\, and Trevor Wilson\footnote{Department of Mathematics, Miami University, Oxford, OH USA. Email: twilson@miamioh.edu.}}
\date{\today}

\begin{document}
\maketitle
\begin{abstract}
    This paper makes significant progress towards resolving a conjecture relating strong forcing axioms like $PFA$ and the derived model at a limit of Woodin cardinals $\kappa$. In particular, using a concept called Covering Matrices, we show that the $\Theta$ of the derived model at $\kappa$ is strictly less than $\kappa^+$ under various circumstances; in particular, this shows that the conclusion holds under $PFA$ if $\kappa$ is a limit of Woodin cardinals of cofinality $\omega$ and the derived model does not satisfy $LSA$. Assuming a form of mouse capturing, we show that the derived model satisfies $AD_{\mathbb{R}}$ under $PFA$ when $\kappa$ is a regular limit of Woodin cardinals. If $\kappa$ is an indestructibly $(\kappa,\kappa^+)$-weakly compact limit of Woodin cardinals, then the derived model outright satisfies $AD_{\mathbb{R}}$.
\end{abstract}
\thispagestyle{firstpage}

\section{Introduction}

Derived models are an important class of models studied extensively by modern set theorists. These are the models of $AD^+$ constructed from large cardinals; conversely, if $AD^+$ holds, then one can show that $L(P(\mathbb{R}))$ is a derived model. These seminal results, due to H. W. Woodin, show essentially that the theory $AD^+$ is a completion of $AD$. Forcing axioms are generalizations of the Baire Category theorem that have found many applications in set theory. Deep work in set theory in the last 30 years have established tight connections between forcing axioms and determinacy, e.g. \cite{woodin2010axiom, steel2005pfa, lsa}.

It is a natural question to ask how strong forcing axioms influence inner models of $AD^+$, like derived models. A well-known conjecture (\cite[Problem 8]{woodin2010axiom}) by Woodin conjectures that if $MM(c)$ holds and $M$ is an inner model of $AD^+$ such that $\Theta^M = \omega_3$, then $M\models AD_{\mathbb{R}}$. The following conjecture follows the same vein, but instead, we ask fundamental questions about the derived models if a global forcing axiom like $PFA$ holds. In the following, we write $D(V,\kappa)$ for the ``new" derived model at $\kappa$ (see \cite{derivedmodelpfa} for a precise definition) and $\Theta$ is the supremum of ordinals $\alpha$ such that there is a surjection from $\mathbb{R}$ onto $\alpha$.

\begin{conjecture}[PFA]\label{conj:PFA}
    Suppose $\kappa$ is a limit of Woodin cardinals. Then the following hold.
    \begin{enumerate}
        \item $\Theta^{D(V,\kappa)} < \kappa^+$.
        \item $D(V,\kappa)\models AD_{\mathbb{R}}.$
    \end{enumerate}
\end{conjecture}

This paper presents some progress towards resolving Conjecture \ref{conj:PFA}. Conjecture \ref{conj:PFA} is a generalization of Wilson's conjecture, which was partly resolved by work of \cite{derivedmodelpfa} under additional mouse capturing assumptions. In this paper, we attempt to study derived models under strong forcing axioms (like PFA) or large cardinals in a more general setting. 
The main theorem of Section \ref{sec:CM} is \ref{result from cp}, which implies part (1) of the conjecture in the case $\kappa$ has cofinality $\omega$, if the derived model does not satisfy $LSA$. The proof utilizes a combinatorial object called a Covering Matrix. Section \ref{sec:CM} addresses part (1) for other values of $\kappa$ as well, but it assumes combinatorial properties of the Covering Matrix not known to follow from $PFA$ (though they are known to follow from large cardinals, \cite{viale}). It is not known if the derived model at any limit of Woodin cardinals can satisfy $LSA$, assuming $PFA$.

In Section \ref{sec:WC}, we show that if $\kappa$ is a limit of Woodin cardinals and $\kappa$ is $Col(\kappa,\kappa^+)$-indestructibly weakly compact, then $D(V,\kappa)\models AD_{\mathbb{R}}$ (Corollary \ref{cor:WC}). We note that it has been shown by Woodin that if $\kappa$ is a limit of Woodin cardinals and $<\kappa$-strong cardinals, then $D(V,\kappa)\models AD_{\mathbb{R}}$; in our situation, $\kappa$ need not be a limit of $<\kappa$-strong cardinals.

The techniques of Section \ref{sec:WC} are used to address partially part (2) of the conjecture under $PFA$ in Section \ref{sec:ADRPFA}. The main theorems of this section are \ref{cof at least kappa} and \ref{derived model not Lp^Sigma}. We note that the work in Section \ref{sec:ADRPFA}, although more general than that in \cite{derivedmodelpfa}, still assumes a form of mouse capturing. We hope to address the question of whether the derived model always satisfies $AD_{\mathbb{R}}$ under $PFA$ without additional inner model theoretic assumptions in a future publication.

\section{Covering Matrices}\label{sec:CM}

Wilson used coherent covering matrices to show:

\begin{theorem}[Wilson]
    Suppose $\kappa$ is a limit of Woodin cardinals of cofinality $\omega$ and there is no coherent covering matrix for $\kappa^+$. Then $\Theta_0^{D(V,\kappa)} < \kappa^+$.
\end{theorem}

Our next theorem generalizes Wilson's result. The non-existence of a coherent covering matrix used in Wilson's argument is a special case of a covering property introduced by Viale in \cite{viale}.

\begin{definition}[Viale]
\label{covering matrix definition}
    Let $\lambda < \kappa$ be regular cardinals. We say $K = \langle K(\alpha,\beta):\alpha < \lambda, \beta<\kappa\rangle$ is a $\lambda$-covering matrix for $\kappa$ if the following hold:
    \begin{enumerate}
        \item $\beta = \bigcup_{\alpha<\lambda} K(\alpha,\beta)$ for all $\beta<\kappa$.
        \item $K(\alpha,\beta) \subseteq K(\eta,\beta)$ for all $\beta<\kappa$ and $\alpha<\eta<\lambda$. Moreover, for all sufficiently large $\beta<\kappa$ and all $\alpha<\lambda$, there is $\eta<\lambda$ such that $K(\alpha,\beta)\subsetneq K(\eta,\beta)$.\footnote{This is not quite the same as Definition 4 of \cite{viale}. \cite{viale} also ought to have said proper containment only holds for large enough $\beta$ to avoid the trivial cases where $\beta <\lambda$. Our other change is equally superficial --- since $\lambda$ is regular, if a sequence $\langle K(\alpha,\beta): \alpha < \lambda\rangle$ satisfies our condition, we could replace it by a subsequence of length $\lambda$ which is strictly increasing.}
        \item For all $\gamma < \beta <\kappa$ and all $\alpha < \lambda$, there is $\eta < \lambda$ such that $K(\alpha,\gamma) \subseteq K(\eta,\beta)$.
        \item For all $X\in [\kappa]^{\leq\lambda}$, there is $\gamma_X < \kappa$ such that for all $\beta < \kappa$ and $\eta <\lambda$, there is $\alpha<\lambda$ such that $K(\eta,\beta) \cap X \subseteq K(\alpha,\gamma_X)$.
    \end{enumerate}
\end{definition}

\begin{definition}[Viale]
    We say $\kappa$ has the $\lambda$-covering property, and write $CP(\kappa,\lambda)$, if for every $K$ a $\lambda$-covering matrix for $\kappa$, there is an unbounded set $A\subseteq\kappa$ such that $[A]^\lambda$ is covered by $K$.\footnote{I.e. for any $X\in[A]^\lambda$, there is $\alpha<\lambda$ and $\beta<\kappa$ such that $X\subseteq K(\alpha,\beta)$.}
\end{definition}

\begin{theorem}
\label{result from cp}
    Suppose $\kappa$ is a limit of Woodin cardinals, $cof(\kappa) = \lambda$, $CP(\kappa^+,\lambda)$ holds,\footnote{By Lemma 14 of \cite{viale} this implies $\lambda < \kappa$.} and $D(V,\kappa) \nvDash LSA$. Then $\Theta^{D(V,\kappa)} < \kappa^+$.
\end{theorem}
\begin{proof}
    Suppose not. Let $D = D(V,\kappa)$ be the derived model at $\kappa$ constructed from generic $G\subseteq Col(\omega,<\kappa)$. If $D \models AD_\R$, then $\Theta^D < \kappa^+$,\footnote{See Theorem 4.12 of \cite{derivedmodelpfa}.} so we may assume $D \models \Theta = \Theta_{\iota+1}$ for some $\iota$. Since $D\nvDash LSA$, any set of reals in $D$ of Wadge rank $\Theta_\iota$ is Suslin-co-Suslin in $D$, hence in $Hom^*$. Then fix some $\gamma<\kappa$ and $A\in Hom^{V[G\upharpoonright\gamma]}_{<\kappa}$, $A\subseteq \R^{V[G\upharpoonright\gamma]}$, such that $D(V,\kappa) \models w(A^*) = \Theta_\iota$.\footnote{$A^*\subseteq \R^*$ is the unique set such that $A^* =\rho[T]$ whenever $T,U\in V[G\upharpoonright\gamma]$ are trees witnessing $A\in Hom^{V[G\upharpoonright\gamma]}_{<\kappa}$.}
    
    Let $T,U\in V[G\upharpoonright\gamma]$ be trees witnessing $A\in Hom^{V[G\upharpoonright\gamma]}_{<\kappa}$. Fix $Col(\omega,\gamma)$-names $\dot{A}$, $\dot{T}$, and $\dot{U}$ for $A$, $T$, and $U$. We may assume all the properties of $A$, $T$, and $U$ mentioned above are forced to hold by the trivial condition in the forcing $Col(\omega,\gamma)$. Additionally, we may assume, for some $\zeta\in On$, the trivial condition forces $D \models l(A^*) = \check{\zeta}$.\footnote{$l(A)$ is the Lipschitz rank of $A$.} For $q\in Col(\omega,\gamma)$, let $A_q = \dot{A}_{H_q}$, $T_q = \dot{T}_{H_q}$, and $U_q = \dot{U}_{H_q}$, where $H_q = q \cup G\upharpoonright(\gamma\backslash dom(q))$. Note $\emptyset \Vdash D(V,\kappa) \models l(A_q^*) = \check{\zeta}$, so $A_q^*\leq_l A^*$.\footnote{I.e. $A^*_q$ is Lipschitz reducible to $A^*$.}

    \begin{claim}
    \label{lipshcitz reduction uniform}
        For any $q\in Col(\omega, \gamma)$, there is Lipschitz $f:\R^* \to \R^*$ such that $f^{-1}(A^*) = A_q^*$ and $f$ is coded by a real in $V[G\upharpoonright\gamma]$.\footnote{Note there is obviously such a Lipschitz $f$ coded by a real in $\R^*$ --- the content of the claim is that we may take it to be in $V[G\upharpoonright\gamma]$. This will be used to collect the sets $A_q^*$ into a single set in $D$. The proof of this claim is where we use that $D \nvDash LSA$. Without this assumption, we still have for each $q$ some $\gamma_q<\kappa$ such that a real in $V[G\upharpoonright\gamma_q]$ codes a reduction of $A_q^*$ to $A^*$, but we don't see how to bound the $\gamma_q$ strictly below $\kappa$.}
    \end{claim}
    \begin{proof}
        Since $A_q^*\leq_l A^*$ in $D$, there is Lipschitz $f:\R^*\mapsto \R^*$ such that $f^{-1}(A^*)=A_q^*$ and $f$ is coded by a real in $\R^*$. Then by Theorem 2.2 of \cite{dmt}, there is Lipschitz $f:\R^{V[G\upharpoonright\gamma]} \to \R^{V[G\upharpoonright\gamma]}$ coded by a real in $\R^{V[G\upharpoonright\gamma]}$ such that $f^{-1}(A) = A_q$. 
        
        Since $f$ is Lipschitz, we may consider $f$ as a function $f:2^{<\omega}\to 2^{<\omega}$ such that $lh(t) = t$ for any $t\in 2^{<\omega}$. Let $f^{-1}(T) =\{(t,s)\in\omega^{<\omega} \times On^{<\omega}: (f(t),s)\in T\}$. Define $f^{-1}(U)$ analogously. It suffices to show $\rho[f^{-1}(T)] \cap \R^* = A_q^*$. Clearly $\rho[f^{-1}(T)]\cap \R^{V[G\upharpoonright\gamma]}= A_q$. Since $A_q^*$ is the unique subset of $\R^*$ such that $A_q^* = \rho[T'] = \rho[U']^c$ for some $<\kappa$-complementing trees $T'$ and $U'$ in $V[G\upharpoonright\gamma]$, it suffices to show that $\rho[f^{-1}(T)] \cup \rho[f^{-1}(U)] = \R^*$. Suppose $a\in\R^*$. Either $f(a)\in A^*$ or $f(a) \in (A^*)^c$. If $f(a) \in A^*$, then there is $\vec{s}\in On^\omega$ such that $(f(a),\vec{s})\in [T]$. But then $(a,\vec{s})\in [f^{-1}(T)]$, so $a\in \rho[f^{-1}(T)]$. Similarly, if $f(a)\in (A^*)^c = \rho[U]$, then $a\in \rho[f^{-1}(U)]$.
    \end{proof}

    By Claim \ref{lipshcitz reduction uniform}, for each $q\in Col(\omega,\gamma)$ there is $x_q\in\R^{V[G\upharpoonright\gamma]}$ coding a Lipschitz reduction of $A_q^*$ to $A^*$. Then there is $\gamma'<\kappa$ and $y\in \R^{V[G\upharpoonright\gamma']}$ such that $y$ codes the function $q\mapsto x_q$ with domain $Col(\omega,\gamma)$. Let $B = \{(x,z) \in \R^*\times \R^*: x \text{ codes } H_q \wedge z\in A_q^*\}$. Then $B$ is in $D$, since $A^*\in D$ and the maps $q\mapsto x_q$ and $q \mapsto H_q$ are in $D$. Clearly $w(B)\geq \Theta_\iota$. So for every $\beta < \kappa^+$, there is a surjection $f_\beta:\R^* \mapsto \beta$ which is ordinal definable from $B$ in $D$. We can choose $f_\beta$ such that $\vec{f} = \langle f_\beta: \beta < \kappa^+\rangle$ is ordinal definable from $B$ in $D$. Pick $g:\lambda \to \kappa$ in $V$ such that $g$ is increasing and cofinal in $\kappa$. For $\alpha < \lambda$ and $\beta <\kappa^+$, let $K(\alpha,\beta) = f_\beta''[\R^{V[G\upharpoonright g(\alpha)]}]$. Let $K = \langle K(\alpha,\beta): \alpha<\lambda, \beta < \kappa^+\rangle$.

    We want to show $K \in V$. Let $\dot{B} \in V$ be a natural name for $B$. That is, we could take $\dot{B}$ to be the set of pairs $(\sigma,p)$ where $p\in Col(\omega,<\kappa)$ and $\sigma$ is a standard $Col(\omega,<\kappa)$-name for a real such that $p$ forces there is $q\in Col(\omega,\gamma)$ such that
    \begin{enumerate}
        \item $\sigma_1$ codes $\dot{H_q}$ and\footnote{$\sigma_1$ and $\sigma_2$ here are some natural names for the pair of reals coded by $\sigma$.}
        \item $\sigma_2\in \dot{A_q}$.\footnote{where $\dot{H_q}$ and $\dot{A_q}$ are names for $H_q$ and $A_q$, respectively, such that the empty condition forces $\dot{H_q} = q\cup \dot{G}\upharpoonright (\gamma\backslash dom(q))$ and $\dot{A_q} = \rho[\dot{T}_{H_q}] \cap \R^*$.}
    \end{enumerate}

    \begin{claim}
    \label{symmetric name}
        If $\alpha < \lambda$ and $\epsilon < \beta < \kappa^+$ then either $\emptyset \Vdash^V_{Col(\omega,<\kappa)} \check{\epsilon} \in f_\beta[\R^{V[\dot{G}\upharpoonright g(\alpha)]}]$ or $\emptyset \Vdash^V_{Col(\omega,<\kappa)} \check{\epsilon} \notin f_\beta[\R^{V[\dot{G}\upharpoonright g(\alpha)]}]$.\footnote{Here $f_\beta$ really refers to the definition in $D$ from $B$ and ordinal parameters.}
    \end{claim}
    \begin{proof}
        Suppose not. Then fix $p,q\in Col(\omega,<\kappa)$ such that $p \Vdash^V_{Col(\omega,<\kappa)} \check{\epsilon} \in f_\beta[\R^{V[\dot{G}\upharpoonright g(\alpha)]}]$ and $q \Vdash^V_{Col(\omega,<\kappa)} \check{\epsilon} \notin f_\beta[\R^{V[\dot{G}\upharpoonright g(\alpha)]}]$. Extending $p$ and $q$ if necessary, we may assume $dom(p) = dom(q)$ and $p\in G$. Let $G_q = q \cup G\upharpoonright (\kappa\backslash dom(q))$. Note $G_q \in D$.

        \begin{subclaim}
            $\dot{B}_G = \dot{B}_{G_q}$
        \end{subclaim}
        \begin{proof}
            Suppose $(x_1,x_2)\in \dot{B}_G$. Then there is $r\in Col(\omega,\gamma)$ such that $x_1$ codes $H_r$ and $x_2\in A_r^* = \rho[\dot{T}_{H_r}] \cap \R^*$. But there is $r'\in Col(\omega,\gamma)$ such that $r'\cup G_q\upharpoonright(\gamma\backslash dom(r')) = H_r$. So $x_1$ codes $r'\cup G_q\upharpoonright(\gamma\backslash dom(r'))$ and $x_2 \in \rho[\dot{T}_{r'\cup G_q\upharpoonright(\gamma\backslash dom(r'))}]\cap\R^*$. Inspecting the definition of $\dot{B}$, it is clear that $(x_1,x_2)\in \dot{B}_{G_q}$.
        \end{proof}

        Since $G$ and $G_q$ disagree at only finitely many values, $\R^{V[G\upharpoonright g(\alpha)]} = \R^{V[G_q\upharpoonright g(\alpha)]}$. But then the definition of $f_\beta\upharpoonright \R^{V[\dot{G}\upharpoonright g(\alpha)]}$ (from $B$ and ordinals) gives the same function in $V[G]$ and $V[G_q]$. Contradiction.
    \end{proof}

    \begin{claim}
        $K\in V$
    \end{claim}
    \begin{proof}
        Immediate from Claim \ref{symmetric name}.
    \end{proof}

    \begin{claim}
    \label{K is covering matrix}
        $K$ is a $\lambda$-covering matrix for $\kappa^+$.
    \end{claim}
    \begin{proof}
        We must check the conditions of Definition \ref{covering matrix definition}.

        \begin{enumerate}
            \item \label{beta covered} For any $\beta < \kappa^+$, $\beta = f_\beta''[\R^*] =  \bigcup_{\alpha<\lambda} f_\beta''[\R^{V[G\upharpoonright g(\alpha)]}] = \bigcup_{\alpha<\lambda} K(\alpha,\beta)$.
            \item The first part holds because $g$ is increasing. For the second part, suppose $\kappa \leq \beta < \kappa^+$ and $\alpha < \lambda$. Note $\beta$ is uncountable in $D$, whereas $\R^{V[G\upharpoonright g(\alpha)]}$ is countable in $D$. So $K(\alpha,\beta) = f_\beta''[\R^{V[G\upharpoonright g(\alpha)]}] \subsetneq \beta$. Then by \ref{beta covered}, there is some $\eta < \lambda$ such that $K(\alpha,\beta) \neq K(\eta,\beta)$.
            \item \label{1st coherence property} Fix $\xi < \beta < \kappa^+$ and $\alpha < \lambda$. Note $\R^{V[G\upharpoonright g(\alpha)]}$ is a countable set of reals in $D$ and $f_\xi \in D$, so $f_\xi''[\R^{V[G\upharpoonright g(\alpha)]}]$ is a countable set of ordinals in $D$. Then since $D$ satisfies countable choice, we can pick $\langle y_i : i<\omega \rangle$ such that $f_\beta''[\langle y_i\rangle] = f_\xi''[\R^{V[G\upharpoonright g(\alpha)]}]$. Pick $\eta < \lambda$ large enough that $\langle y_i\rangle \subseteq V[G\upharpoonright\eta]$. Then $K(\alpha,\xi) \subseteq K(\eta,\beta)$.
            \item In fact, we can prove the stronger property (iv') from \cite{viale}: the same proof as for property \ref{1st coherence property} gives for all $\xi < \beta <\kappa^+$  and all $\eta < \lambda$, there is $\alpha < \lambda$ such that $K(\eta,\beta) \cap \xi \subseteq K(\alpha,\xi)$.
        \end{enumerate}
    \end{proof}

    \begin{claim}
    \label{ot less than kappa}
        For all $\alpha < \lambda$ and $\beta<\kappa^+$, $o.t.(K(\alpha,\beta))<\kappa$.
    \end{claim}
    \begin{proof}
        $K(\alpha,\beta)$ is the image of a countable set of reals in $D$ under $f_\beta$. So $K(\alpha,\beta)$ is countable in $D$, whereas $\kappa = \omega_1^D$.
    \end{proof}

    Since $K\in V$, $CP(\kappa^+,\lambda)$ gives an unbounded $A\subseteq \kappa^+$ such that $[A]^\lambda$ is covered by $K$. Pick $\tau < \kappa^+$ such that $o.t.(A\cap \tau)\geq\kappa$. By Claim \ref{ot less than kappa}, for each $\alpha <\lambda$, there is $\xi_\alpha\in A\cap \tau \backslash K(\alpha,\tau)$. Let $X = \{\xi_\alpha:\alpha <\lambda\}$. $X\in [A]^\lambda$, so the following claim gives a contradiction.

    \begin{claim}
        $X$ is not covered by $K(\alpha,\beta)$ for any $\alpha<\lambda$, $\beta<\kappa^+$.
    \end{claim}
    \begin{proof}
        First suppose $\beta \leq \tau$. Since $K$ is a covering matrix, there is $\eta < \lambda$ such that $K(\alpha,\beta)\subseteq K(\eta,\tau)$. Then $\xi_\eta \in X\backslash K(\alpha,\beta)$.

        Now suppose $\tau <\beta$. Applying the property (iv') we proved in Claim \ref{K is covering matrix}, there is $\eta <\lambda$ such that $K(\alpha,\beta) \cap \tau \subseteq K(\eta,\tau)$. So $\xi_\eta \in X\backslash K(\alpha,\beta)$.
    \end{proof}
\end{proof}

\begin{corollary}
    Suppose $\kappa$ is a limit of Woodin cardinals, $\tau$ is strongly compact, $\kappa \geq \tau$, $cof(\kappa) = \lambda < \tau$, and $D(V,\kappa) \nvDash LSA$. Then $\Theta^{D(V,\kappa)} < \kappa^+$.
\end{corollary}
\begin{proof}
    Theorem 10 of \cite{viale} implies $CP(\kappa^+,\lambda)$. Then apply Theorem \ref{result from cp}.
\end{proof}

\begin{corollary}
    Asssume $PFA$. If $\kappa$ is a limit of Woodin cardinals of cofinality $\omega$ and $D(V,\kappa)\nvDash LSA$, then $\Theta^{D(V,\kappa)} < \kappa^+$.
\end{corollary}
\begin{proof}
    By \cite{viale}, $PFA$ implies $CP(\kappa^+,\omega)$. Then apply Theorem \ref{result from cp}.
\end{proof}

Whether Theorem \ref{result from cp} is a viable method towards proving all of part (1) of Conjecture \ref{conj:PFA} remains to be seen ---  besides resolving the case $D(V,\kappa)\models LSA$, it is unknown whether $PFA$ (or even $MM$) implies $CP(\kappa^+,cof(\kappa))$ if $cof(\kappa) >\omega$.

\section{Weakly compact cardinals}\label{sec:WC}

The following theorem appeared in Wilson's thesis (\cite{trevorthesis}). In the next section we will use its proof, which we include here in greater detail for the reader's convenience.

\begin{theorem}[Wilson]
\label{weakly cpt thm}
    Suppose $\kappa$ is weakly compact and $\kappa$ is a limit of Woodin cardinals. If $D(V,\kappa)\nvDash AD_\R$, then $\Theta^{D(V,\kappa)} = \kappa^+$.
\end{theorem}
\begin{proof}
    Suppose not. Take $N\prec V$ such that $N$ is transitive, $|N| = \kappa$, $\Theta^{D(V,\kappa)} \subset N$, and $^{<\kappa}N\subset N$. By weak compactness of $\kappa$, there is an elementary embedding $j: N \to M$ such that $M$ is transitive, $|M| = \kappa$, $^{<\kappa}M\subset M$ and $crit(j) = \kappa$.

    Let $G\subseteq Col(\omega,<\kappa)$ be generic over $V$ (in particular, $G$ is generic over $N$). Let $D$ be the derived model of $N$ at $\kappa$ constructed using $G$. As we are assuming the theorem fails, $D\nvDash AD_\R$. Then $D$ has a largest Suslin pointclass $\mathbf{\Gamma}$, and corresponding largest Suslin cardinal $\delta$. Let $T\in D$ be a tree on $\omega \times \delta$ which is a tree for a scale on a universal $\mathbf{\Gamma}$-set. We may assume $T\in V$.

    We can extend $j$ to a map $j: N[G]\to M[G*H]$, where $H$ is chosen such that $G*H \subseteq Col(\omega,<j(\kappa))$ is generic over $V$. Note then $j(D)$ is the derived model in $M$ at $j(\kappa)$.

    Let $\sigma = j''((meas^\mathbf{\Gamma}(\delta^{<\omega}))^D)$.

    \begin{claim}
    \label{sigma in M(R)}
        $\sigma\in HOD^{M[G*H]}_{M\cup \R^*_{G*H}}$ and is countable in $HOD^{M[G*H]}_{M\cup \R^*_{G*H}}$.
    \end{claim}
    \begin{proof}
        Clearly $\sigma \subset j(D) \subset HOD^{M[G*H]}_{M\cup \R^*_{G*H}}$. Since we are assuming $\Theta^{D(V,\kappa)}<\kappa^+$, $\Theta^D < (\kappa^+)^N$, so there is a surjection $f: \kappa \to (meas^{\mathbf{\Gamma}}(\delta^{<\omega}))^D$ in $N[G]$. Then, since $crit(j) = \kappa$, $j(f) \upharpoonright \kappa$ is a surjection of $\kappa$ onto $\sigma$ in $M[G*H]$. And $\kappa$ is countable in $M[G*H]$, so $\sigma$ is countable in $M[G*H]$.

        Enumerate $\sigma$ in $M[G*H]$ as $\sigma = \langle A_i:i\in\omega \rangle$. Pick $x_i\in \R^*_{G*H}$ such that $A_i$ is definable in $M[G*H]$ from parameters in $M\cup\{x_i\}$. Since $M^\omega \subset M$, by taking the union of the parameters in $M$ used to define $A_i$, there is a single set $X \in M$ such that for any $i\in\omega$, $A_i$ is definable in $M[G*H]$ from $X$ and $x_i$.

        Let $x = \bigoplus_{i\in\omega} x_i$. Since $j(\kappa)$ is regular in $M$, $cof(j(\kappa))^{M[G*H]} > \omega$, so $x\in \R^*_{G*H}$. And $\sigma$ is definable in $M[G*H]$ from $X$ and $x$, so $\sigma \in HOD^{M[G*H]}_{M\cup\R^*_{G*H}}$.
    \end{proof}

    Let $\tilde{T} = j(T)$. Consider the game $G^\sigma_{\tilde{T}}$ from Definition 3.6.3 of \cite{trevorthesis}

    \begin{claim}
    \label{PII wins}
        Player II has a winning strategy in $HOD^{M[G*H]}_{M\cup \R^*_{G*H}}$ for the game $G^\sigma_{\tilde{T}}$.
    \end{claim}
    \begin{proof}
        We alter the game $G^\sigma_{\tilde{T}}$ to get an equivalent game in $V$ as follows. By Proposition 3.5.5 of \cite{trevorthesis}, there is a wellordering $\leq^* \in j(D)$ of $j((meas^{\mathbf{\Gamma}}(\delta)^{<\omega})^D)$ which is definable in $V[G*H]$ from $\tilde{T}$. And $\sigma$ is also definable in $V[G*H]$ from $j\upharpoonright N$, so the set $Z$ of ordinals coding sets in $\sigma$ relative to the wellordering $\leq^*$ is in $V$. Let $G_{\leq^*}$ be the game played just like $G^\sigma_{\tilde{T}}$, except Player II plays elements of $Z$ coding a measure in $\sigma$ rather than the measure itself. Then $G_{\leq^*}$ is in both $V$ and $HOD^{M[G*H]}_{M\cup \R^*_{G*H}}$. $G_{\leq^*}$ is a closed game, so Player II has a winning strategy for $G_{\leq^*}$ in $HOD^{M[G*H]}_{M\cup \R^*_{G*H}}$ if and only if Player II has a winning strategy for $G_{\leq^*}$ in $V$. Clearly, playing $G_{\leq^*}$ and $G^\sigma_{\tilde{T}}$ in $HOD^{M[G*H]}_{M\cup \R^*_{G*H}}$ are equivalent, so it suffices to show Player II has a winning strategy for $G_{\leq^*}$ in $V$.
        
        We define a winning strategy $\tau\in V$ for Player II in the game $G_{\leq^*}$ as follows. Suppose $\langle (n_j,\alpha_j,h_j,\zeta_j): j \leq i\rangle$ is a partial play of $G_{\leq^*}$, where $\zeta_j$ codes some $\mu_j\in \sigma$. For $A\in P(\delta)^D$, let $A\in \bar{\mu}_{i+1}$ if and only if $\langle \alpha_0,...,\alpha_i\rangle \in j(A)$. Let $\mu_{i+1} = j(\bar{\mu}_{i+1})$.

        \begin{subclaim}
        \label{mu in derived model}
            $\bar{\mu}_{i+1}\in (meas^\mathbf{\Gamma}(\delta^{<\omega}))^D$ (and therefore $\mu_{i+1} \in \sigma$).
        \end{subclaim}
        \begin{proof}
            Clearly, $\bar{\mu}_{i+1}$ is countably complete --- what we must show is that $\bar{\mu}_{i+1}\in D$. Since $\kappa$ is regular, $\R^{V[G]} = \R^*_G$. Then $\mathbf{\Gamma}$ is also an inductive-like pointclass of $V[G]$ and $V[G] \models \mathbf{\Delta_\Gamma}$ is determined. Then by Lemma 3.5.4 of \cite{trevorthesis} (applied in $V[G]$), the code set for $\bar{\mu}_{i+1}$ is in $\mathbf{Env(\Gamma)}^{V[G]}$. But $\mathbf{Env(\Gamma)}^{V[G]} = \mathbf{Env(\Gamma)}^D$,\footnote{By Proposition 3.2.5 of \cite{trevorthesis}, any set in $\mathbf{Env(\Gamma)}$ is ordinal definable from a universal $\mathbf{\Gamma}$-set and a real in $\R^*_G$, so $\mathbf{Env(\Gamma)}\subseteq V(\R^*_G)$. From the definition of the envelope, we get any two transitive models with the same reals and ordinals and both containing $\mathbf{\Gamma}$ will agree on $\mathbf{Env(\Gamma)}$. This immediately gives $\mathbf{Env(\Gamma)}^D = \mathbf{Env(\Gamma)}^{N[G]}$. To get $\mathbf{Env(\Gamma)}^{N[G]} = \mathbf{Env(\Gamma)}^{V[G]}$, we use that we are assuming for contradiction that $\Theta^{D(V,\kappa)} < \kappa^+$. Because we picked $N$ such that $\Theta^{D(V,\kappa)} \subset N$, $\mathbf{Env(\Gamma)}^{N[G]}$ is a Wadge initial segment of $\mathbf{Env(\Gamma)}^{V[G]}$ with the same Wadge rank, so they are equal.} so $\bar{\mu}_{i+1}\in D$.
        \end{proof}

        \begin{subclaim}
        \label{mu projects correctly}
            $proj_{i+1,i}(\mu_{i+1}) = \mu_i$
        \end{subclaim}
        \begin{proof}
            It suffices to show $proj_{i+1,i}(\bar{\mu}_{i+1}) = \bar{\mu}_i$. But\footnote{See Definition 3.5.6 of \cite{trevorthesis} for the definitions of $proj$ and $ext$.}
            \begin{align*}
                A\in \bar{\mu}_i &\implies \langle \alpha_0,...,\alpha_{i-1}\rangle \in j(A)\\
                &\implies \langle \alpha_0,...,\alpha_i\rangle\in ext_{i,i+1}(j(A))\\
                &\implies \langle \alpha_0,...,\alpha_i\rangle \in j(ext_{i,i+1}(A))\\
                &\implies ext_{i,i+1}(A) \in \bar{\mu}_{i+1}.
            \end{align*}
        \end{proof}

        \begin{subclaim}
        \label{mu concentrates on T}
            $\mu_{i+1}$ concentrates on $\tilde{T}_{(n_0,...,n_i)}$.
        \end{subclaim}
        \begin{proof}
            We must show $\{(\xi_0,...,\xi_i):((n_0,...,n_i),(\xi_0,...,\xi_i))\in\tilde{T}\}\in \mu_{i+1}$. But $\{(\xi_0,...,\xi_i):((n_0,...,n_i),(\xi_0,...,\xi_i))\in\tilde{T}\} = j(\{(\xi_0,...,\xi_i):((n_0,...,n_i),(\xi_0,...,\xi_i))\in T\})$. So it suffices to show $\{(\xi_0,...,\xi_i):((n_0,...,n_i),(\xi_0,...,\xi_i))\in T\} \in \bar{\mu}_{i+1}$, i.e. $(\alpha_0,...,\alpha_i)\in j(\{(\xi_0,...,\xi_i):((n_0,...,n_i),(\xi_0,...,\xi_i))\in T\})$. But $j(\{(\xi_0,...,\xi_i):((n_0,...,n_i),(\xi_0,...,\xi_i))\in T\}) = \{(\xi_0,...,\xi_i): ((n_0,...,n_i),(\xi_0,...,\xi_i))\in \tilde{T}\}$, and $(\alpha_0,...,\alpha_i)\in \tilde{T}_{(n_0,...,n_i)}$ by the rules of $G_{\leq^*}$.
        \end{proof}

        Let $\tau(\langle (n_j,\alpha_j,h_j,\zeta_j): j\leq i\rangle) = \zeta_{i+1}$, where $\zeta_{i+1}$ is the unique ordinal coding $\mu_{i+1}$ relative to $\leq^*$. Subclaims \ref{mu in derived model}-\ref{mu concentrates on T} show that $\tau$ plays valid moves for Player II in $G_{\leq^*}$. If both players follow the rules of $G_{\leq^*}$, then Player I wins. So to show $\tau$ is a winning strategy, we suppose $\langle n_i,\alpha_i,h_i,\zeta_i:i\in\omega\rangle$ is a play of $G_{\leq^*}$ in which Player II has played according to $\tau$ and Player I has played in accordance with the rules of $G_{\leq^*}$ and derive a contradiction.

        By the rules for Player I, $j_{\mu_i,\mu_{i+1}}(h_i) > h_{i+1}$. By the countable closure of $M$, $\langle \zeta_i:i<\omega\rangle \in M$, and thus $\vec{\mu} = \langle \mu_i:i\in\omega\rangle \in M[G*H]$. Also by countable closure of $M$, $\langle h_i:i\in\omega\rangle \in M$. So $\langle h_i:i\in\omega\rangle$ witnesses $\vec{\mu}$ is illfounded in $M[G*H]$. 
        
        \begin{claim}
            $\bar{\mu} = \langle \bar{\mu}_i:i\in\omega\rangle$ is illfounded in $N[G]$.
        \end{claim}
        \begin{proof}
            If $\bar{\mu} \in N[G]$, then $j(\bar{\mu}) = \vec{\mu}$ (since $j(\bar{\mu}_i) = \mu_i$ for each $i$) and the claim follows from elementarity of $j$. So it suffices to show $\bar{\mu} \in N[G]$. 
            
            Since $j(\bar{\mu}_i) = \mu_i$, there is $\bar{\zeta}_i \in N$ such that $j(\bar{\zeta}_i) = \zeta_i$ and $\bar{\zeta}_i$ codes $\bar{\mu}_i$ relative to $j^{-1}(\leq^*)$. But $\langle \zeta_i:i<\omega\rangle \in V$, so $\langle \bar{\zeta}_i:i<\omega\rangle \in V$ and by the countable closure of $N$, $\langle \bar{\zeta}_i:i<\omega\rangle \in N$. It follows that $\bar{\mu} \in N[G]$.
        \end{proof}

        On the other hand, suppose $\langle X_i :i\in\omega\rangle\in N[G]$ is such that $X_i\subseteq P(\delta^i)$ and $X_i\in\bar{\mu}_i$. To show $\bar{\mu}$ is wellfounded in $N[G]$, we need to show in $N[G]$ there is $f\in\delta^\omega$ such that for every $i\in\omega$, $f\upharpoonright i\in X_i$. Again by elementarity, it suffices to show in $M[G*H]$ there is $f\in j(\delta)^\omega$ such that $f\upharpoonright i \in j(X_i)$ for each $i\in\omega$. $f(i) = \alpha_i$ works, since $\langle \alpha_i:i\in\omega\rangle\in M[G*H]$ by the countable closure of $M$. But then $\bar{\mu}$ is wellfounded in $N[G]$, a contradiction.
    \end{proof}

    From Claim \ref{PII wins}, and Lemma 3.6.4 of \cite{trevorthesis}, there is a tree $S \in HOD^{M[G*H]}_{M\cup \R^*_{G*H}}$ such that $\rho[S] = \rho[\tilde{T}]^c$. Then by elementarity of $j$, $HOD^{N[G]}_{N\cup \R^*_G}$ satisfies there is a tree $S$ such that $\rho[S] = \rho[T]^c$. Then there is $\gamma < \kappa$ such that $S\in V[N\upharpoonright \gamma]$. But then in $D$, $\rho[T]^c$ is Suslin, a contradiction.
\end{proof}

\begin{corollary}[Wilson]\label{cor:WC}
    If $\kappa$ is a limit of Woodins and $\kappa$ is $Col(\kappa, \kappa^+)$-indestructibly weakly compact, then $D(V,\kappa)\models AD_\R$.
\end{corollary}

\section{$AD_\R$ from PFA}\label{sec:ADRPFA}

In the last section we saw $D(V,\kappa) \models AD_\R$ follows from sufficient large cardinal hypotheses on $\kappa$. Here we provide some evidence that assuming $PFA$, weaker assumptions on $\kappa$ should give $D(V,\kappa)\models AD_\R$. We do not see how to prove the results of this section without assuming a version of mouse capturing holds in the derived model. Specifically, we will assume there is a hod pair $(P,\Sigma)$ in the derived model such that the derived model satisfies
\begin{enumerate}
        \item \label{smsc} $V = L(Lp^\Sigma(\R))$ and
        \item \label{super-small msc} super-small $\Gamma$-$\Sigma^*$-mouse capturing holds on a cone (in the sense of Definition 5.25 of \cite{lsa}, where $\bf{\Gamma}$ is the largest Suslin pointclass of the derived model).
\end{enumerate}

\begin{remark}
    Assumptions (\ref{smsc}) and (\ref{super-small msc}) may be redundant. (\ref{smsc}) follows from (\ref{super-small msc}) assuming some smallness assumptions (see \cite{mscfsor}). On the other hand, Remark \ref{proving super-small mc} shows we can replace our use of (\ref{super-small msc}) in the proof of Theorem \ref{cof at least kappa} with another application of (\ref{smsc}) if we are working below superstrong cardinals.
\end{remark}

\begin{theorem}
\label{cof at least kappa}
    Suppose $\kappa$ is a regular limit of Woodin cardinals and there is a hod pair $(P,\Sigma)$ in $D(V,\kappa)$ such that $D(V,\kappa)$ satisfies $V = L(Lp^\Sigma(\R^*))$ and ``super-small $\Gamma$-$\Sigma^*$-mouse capturing holds on a cone.'' Then $cof(\Theta^{D(V,\kappa)})\geq \kappa$.
\end{theorem}

\begin{remark}
    Note $V = L(Lp^\Sigma(\R))$ implies any set of reals is $OD(\Sigma,x)$ for some $x\in\R$. In particular, our hypothesis implies $D(V,\kappa) \nvDash AD_\R$. If $\kappa$ is a regular limit of Woodin cardinals and $D(V,\kappa) \models AD_\R$, then $cof(\Theta^{D(V,\kappa)}) = \kappa$.
\end{remark}

\begin{proof}[Proof of Theorem \ref{cof at least kappa}]
    Let $D = D(V,\kappa, G)$ and $\Theta = \Theta^D$. Let $\R^* = \R^D = \bigcup_{\zeta<\kappa} \R^{V[G\upharpoonright\zeta]}$. Suppose the theorem fails. So $D \nvDash AD_\R$ and $cof(\Theta) < \kappa$. Let $\mathbf{\Gamma}$ be the largest Suslin pointclass of $D$ and let $\delta$ be the largest Suslin cardinal of $D$. Let $T\in D$ be a tree projecting to a universal $\mathbf{\Gamma}$-set $U$. Let $(P,\Sigma)$ be a hod pair in $D$ such that $D \models V = L(Lp^\Sigma(\R))$. $T,P,\Sigma \in V[G\upharpoonright \zeta]$ for some $\zeta < \kappa$, but, collapsing $\zeta$ if necessary, we may assume $T,P,\Sigma \in V$.
    Consider $X \prec V_{\kappa^{++}}$ such that
    \begin{enumerate}
        \item $X \cap \kappa = \kappa_X \in On$,
        \item $X^{<\kappa_X} \subset X$,
        \item $X$ is cofinal in $\Theta$, and
        \item $T\in X$.
    \end{enumerate}

    Say $X\prec V_{\kappa^{++}}$ is good if it satisfies the above properties. Let $\pi_X:N_X \to X \prec V$ be the anti-collapse embedding. Since $\kappa$ is regular, $\pi_X$ can be extended to a map $\pi_X: N_X[G\upharpoonright\kappa_X] \to V_{\kappa^{++}}[G]$. Let $D_X = \pi_X^{-1}(D)$, $\Theta_X = \pi_X^{-1}(\Theta)$, $\mathbf{\Gamma}_X = \pi_X^{-1}(\mathbf{\Gamma})$, $T_X = \pi_X^{-1}(T)$ and $\delta_X = \pi_X^{-1}(\delta)$. Let $\R_X = \pi_X^{-1}(\R^*) = \bigcup_{\zeta < \kappa_X} \R^{V[G\upharpoonright\zeta]}$. Note $D_X|\Theta \trianglelefteq (Lp^\Sigma(\R_X))^{V[G]}$.

    \begin{lemma}
    \label{fullness}
        $D_X|\Theta_X =  (Lp^\Sigma(\R_X))^{V[G]}$
    \end{lemma}
    \begin{proof}
        Suppose not. Let $M_X$ be the least initial segment of $(Lp^\Sigma(\R_X))^{V[G]}$ such that $\rho(M_X) = \R_X$ and $M_X\triangleright D_X|\Theta_X$. Note $M_X\models AD^+$. Let $E_X$ be the extender of length $\Theta$ derived from $\pi_X$ and let $M_X^* = Ult(M_X,E_X)$.

        Let $\H_X = \H^{M_X}$. I.e., letting $Th\subset \Theta_X$ be the $\Sigma_1$-theory of $M_X$ with parameters in $\Theta_X \cup \{P\}$, $\H_X$ is the $S$-construction (relative to $\Sigma$) in $M_X$ over $Th$.\footnote{The precise definition may be found in Definition 4.6 of \cite{sihmor}.} By Theorem 4.10 of \cite{sihmor}, $\H_X$ is countably iterable above $\Theta_X$ (in $V[G]$).

        Let $\H^*_X = Ult(\H_X, E_X)$. $\H^*_X$ is wellfounded, since $H_X\in V$, so $\H^*_X$ embeds into $\pi_X(\H_X)$. But $On \cap \H^*_X = On \cap M^*_X$, so this implies $M^*_X$ is wellfounded.

        \begin{claim}
        \label{M^*_X iterable}
            $M_X^*$ is countably iterable in $V[G]$ (as a $\Sigma$-premouse).
        \end{claim}
        \begin{proof}
            Let $Y\prec V_{\kappa^{++}}$ be good. Pick $Z\prec V_{\kappa^{++}}$ such that $\{\H^*_X,\H^*_Y,P,\Sigma\} \subset Z$ and $|Z| < min(\kappa_X,\kappa_Y)$. Let $\pi_Z:N_Z \to Z$ be the anticollapse embedding. Let $\H^*_{X,Z} = \pi_Z^{-1}(\H^*_X)$ and $\H^*_{Y,Z} = \pi_Z^{-1}(\H^*_Y)$.

            \begin{subclaim}
                $H^*_{X,Z}$ and $H^*_{Y,Z}$ are countably iterable above $\Theta_Z = \pi_Z^{-1}(\Theta)$ (as $\Sigma$-premice in $V[G]$).
            \end{subclaim}
            \begin{proof}
                We show the subclaim for $\H^*_{X,Z}$. Since $X^{<\kappa_X} \subset X$, $E_X$ is $<\kappa_X$-complete. Then, for any elementary substructure $U\prec \H^*_X = Ult(\H_X,E_X)$ with $|U|<\kappa_X$, if $f:N_U\to U\prec \H^*_X$ is the anticollapse embedding, then there is $\tau:N_U \to \H_X$ such that $\pi_{E_X}\circ \tau(x) = f(x)$ whenever $x\in N_U$ is such that $f(x) \in range(\pi_{E_X})$.\footnote{This is essentially Lemma 8.12 of \cite{fine_structure}} In particular, since $|Z|<\kappa_X$, there is $\tau:\H^*_{X,Z} \to \H_X$ such that $\tau(\Theta_Z) = \Theta_X$. Then, as $\H_X$ is countably iterable above $\Theta_X$ (in $V[G]$), $\H^*_{X,Z}$ is countably iterable above $\Theta_Z$.
            \end{proof}

            $\H^*_X|\Theta = \H^*_Y|\Theta$, so $\H^*_{X,Z}|\Theta_Z = \H^*_{Y,Z}|\Theta_Z$. Then by the subclaim, either $\H^*_{X,Z} \trianglelefteq \H^*_{Y,Z}$ or $\H^*_{Y,Z} \trianglelefteq \H^*_{X,Z}$. By elementarity of $\pi_Z$, $\H^*_{X} \trianglelefteq \H^*_{Y}$ or $\H^*_{Y} \trianglelefteq \H^*_{X}$. 
            
            Let $Vop_\omega$ be the version of the Vopenka algebra described in Section 3 of \cite{scales_in_k}. Let $\dot{R}$ be as defined on p.185 of \cite{scales_in_k}. Note $Vop_\omega$ and $\dot{R}$ are in $\H^*_X \cap \H^*_Y$. Let $F$ be $Vop_\omega$-generic over both $\H^*_X$ and $\H^*_Y$ such that $\dot{R}^F = \R^*$.\footnote{Lemmas 3.4 and 3.5 of \cite{scales_in_k} imply such an $F$ exists.} 
            
            Suppose that $\H^*_X = \H^*_Y|\xi$ for some $\xi$. By Lemma 4.8 of \cite{sihmor}, $M^*_X$ is definable over $\H^*_X[F]$ from $F$ and $M^*_Y|\xi$ is definable over $\H^*_Y|\xi[F]$ from $F$ by the same definition.\footnote{That the definitions of $M^*_X$ and $M^*_Y|\xi$ in these models are the same is not explicitly stated in Lemma 4.8, but is clear from the proof.} Thus $M^*_X\trianglelefteq M^*_Y$. Similarly, $\H^*_Y\trianglelefteq \H^*_X$ implies $M^*_Y\trianglelefteq M^*_X$.

            We have shown $M^*_X\trianglelefteq M^*_Y$ or $M^*_Y\trianglelefteq M^*_X$. But no proper initial segment of $M^*_X$ extending $D$ projects to $\Theta$ (and similarly for $M^*_Y$), so $M^*_X = M^*_Y$.

            Suppose $N$ is a countable hull of $M^*_X$ (in $V[G]$). Pick $Y\prec V_{\kappa^{++}}$ be such that $X\in Y$, $X\subset Y$, and $Y$ is good. We may also choose $Y$ such that $N$ is countable in $N_Y[G\upharpoonright\kappa_Y]$. Since $M^*_Y = M^*_X$, $N$ is a countable hull of $M^*_Y$. Then by elementarity, and that $\pi_Y(N) = N$, $N$ is a countable hull of $M_Y$. $M_Y$ is countably iterable in $V[G]$, so $N$ is as well.
        \end{proof}

        \begin{claim}
        \label{scales analysis}
            Let $A\subset \R^*$ be definable over $M^*_X$ such that $A\notin M^*_X$. Then $L(A,\R^*)\models AD^+$.
        \end{claim}
        \begin{proof}
            Work in $V[G]$. Claim \ref{M^*_X iterable} implies $M^*_X \triangleleft Lp^\Sigma(\R^*)$. So it suffices to show if $A\in Lp^\Sigma(\R^*)$, then $L(A,\R^*)\models AD^+$. We split into two cases, although we will show the first case cannot occur.

            Case 1: $\textbf{Env}(\mathbf{\Gamma}) \subsetneq Lp^\Sigma(\R^*) \cap P(\R^*)$.

            Let $\beta$ be minimal such that there is $A\subset \R^*$ which is definable over $Lp^\Sigma(\R^*)|\beta$ so that $A\notin \textbf{Env}(\mathbf{\Gamma})$. By Theorem 3.2.4 of \cite{trevorthesis}, $Lp^\Sigma(\R^*)|\beta\models AD$. Let $\beta_0$ be such that $[\delta,\beta_0]$ is a $\Sigma_1$-gap in $Lp^\Sigma(\R^*)$. 

            If $\beta_0 < \beta$, then $Lp^\Sigma(\R^*)|\beta_0+1\models AD$ and $\beta_0+1$ begins a $\Sigma_1$-gap. Then Section 5.1 of \cite{sihmor} shows $Lp^\Sigma(\R^*)|\beta_0+1 \models r\Sigma_1^{Lp^\Sigma(\R^*)|\beta_0+1}$ is scaled.\footnote{This is by Theorem 5.1 of \cite{sihmor} if $Lp^\Sigma(\R^*)|\beta_0+1$ is passive and by Theorem 5.9 in the case that $Lp^\Sigma(\R^*)|\beta_0+1$ is ``$P$-active.'' \cite{sihmor} neglects to mention the case that the beginning of a gap is ``$E$-active,'' but this is by the proof of Theorem 5.1.} In particular, there is a scale for $U^c$ definable over $Lp^\Sigma(\R^*)|\beta_0+1$.

            If $\beta_0 = \beta$, then $[\delta,\beta]$ is an admissible $\Sigma_1$-gap and $Lp^\Sigma(\R^*)|\beta\models AD$. Let $n$ be minimal such that $\rho^{Lp^\Sigma(\R^*)|\beta}_{n+1} = \omega$. If $[\delta,\beta]$ is a weak gap, then by Theorem 5.26 of \cite{sihmor}, $Lp^\Sigma(\R^*)|\beta \models r\Sigma_{n+1}^{Lp^\Sigma(\R^*)|\beta}$ is scaled.\footnote{It is clear the hypotheses of Theorem 5.26 are satisfied. Note in this case $\beta$ ends the $\Sigma_1$-gap $[\delta,\beta]$ and $\Sigma \in Lp^\Sigma|\delta$, since $\Sigma$ is Suslin-co-Suslin.} In particular, a scale for $U^c$ is definable over $Lp^\Sigma(\R^*)|\beta$. $[\delta,\beta]$ cannot be a strong gap, since this would imply $Lp^\Sigma(\R^*)|\beta+1 \cap P(\R^*) \subseteq \textbf{Env}(\mathbf{\Gamma})$,\footnote{See p. 46 of \cite{trevorthesis}. ($2'$) on that page gives the result for a level of the $J$-hierarchy, but the proof works for $Lp^\Sigma$ as well.} contradicting our definition of $\beta$.

            Similarly, if $\beta_0 >\beta$, then $Lp^\Sigma(\R^*)|\beta_0 \cap P(\R^*) \subseteq \textbf{Env}(\mathbf{\Gamma})$, contradicting our choice of $\beta$.

            We have shown there is a scale for $U^c$ definable over an initial segment of $Lp^\Sigma(\R^*)$. In particular, there is a tree $S\in V(\R^*)$ such that $\rho[S] = U^c$. Since $S$ is coded by a set of ordinals and $S\in V(\R^*)$, there is $\gamma < \kappa$ such that $S\in V[G\upharpoonright\gamma]$. Then $T$ and $S$ witness that $U \in Hom^*$. This implies $U$ is Suslin-co-Suslin in $D$, a contradiction.

            Case 2: $\textbf{Env}(\mathbf{\Gamma}) = Lp^\Sigma(\R^*) \cap P(\R^*)$.

            Every set in $\textbf{Env}(\mathbf{\Gamma})$ is determined (see Theorem 3.2.4 of \cite{trevorthesis}). But $L(Lp^\Sigma(\R^*)) \cap P(\R^*) = Lp^\Sigma(\R^*) \cap P(\R^*)$, so $L(Lp^\Sigma(\R^*)) \models AD$. And $L(Lp^\Sigma(\R^*))$ satisfies ``$\Sigma_1$-reflection to Suslin-co-Suslin.''\footnote{For suppose $\phi(u,v)$ is a $\Sigma_1$ formula and $L(Lp^\Sigma(\R)) \models \exists C\subseteq P(\R^*) \phi(C,\R^*)$. Then there is $C \in Lp^\Sigma(\R^*)|(\delta^2_1)^{Lp^\Sigma(\R^*)}$ such that $L(Lp^\Sigma(\R^*)) \models \psi(C,\R^*)$. Such a $C$ is Suslin-co-Suslin (again by the scales analysis of \cite{sihmor}).} But $AD^+$ is equivalent to $ZF\, +\, AD \,+\, V = L(P(\R))\, +$  ``$\Sigma_1$-reflection to Suslin-co-Suslin,'' so $L(Lp^\Sigma(\R^*))\models AD^+$. Any inner model of a model of $AD^+$ containing all its reals also satisfies $AD^+$, so $L(A,\R^*)\models AD^+$ for every $A\in P(\R^*) \cap Lp^\Sigma(\R^*)$.
        
        \end{proof}

        Let $A\subset \R^*$ be as in the claim above. Note $M^*_X\in V(\R^*)$, since $M^*_X \triangleleft Lp^\Sigma(\R^*)$.
        Thus, the claim implies $A \in D$. But $A\notin M^*_X$ and $M^*_X \supseteq D \cap P(\R^*)$, a contradiction.\\
    \end{proof}

    Let $\sigma = \pi_X''((meas^{\mathbf{\Gamma}_X}(\delta_X^{<\omega}))^{D_X}) \subset (meas^{\mathbf{\Gamma}}(\delta^{<\omega}))^D$. Note $\sigma$ is countable in $V[G]$. Then we may consider the game $G^\sigma_T$.
    
    \begin{claim}
        Player II has a winning strategy in $HOD^{V[G]}_{V\cup\R^*}$ for the game $G^\sigma_T$.
    \end{claim}
    \begin{proof}
        We argue as in the proof of Claim \ref{PII wins} of Theorem \ref{weakly cpt thm}. Suppose $\langle (n_j,\alpha_j,h_j,\mu_j): j \leq i\rangle$ is a partial play of $G^\sigma_T$. For $A\in P(\delta_X)^D$, let $A\in \bar{\mu}_{i+1}$ if and only if $\langle \alpha_0,...,\alpha_i\rangle \in \pi_X(A)$. Let $\mu_{i+1} = \pi_X(\bar{\mu}_{i+1})$.

        \begin{subclaim}
        \label{measure is valid play}
            $\bar{\mu}_{i+1}\in (meas^{\mathbf{\Gamma_X}}(\delta_X^{<\omega}))^{D_X}$ (and therefore $\mu_{i+1}\in \sigma$).
        \end{subclaim}
        \begin{proof}
            For ease of notation, we prove the case $\bar{\mu}_1\in (meas^{\mathbf{\Gamma_X}}(\delta_X^{<\omega}))^{D_X}$. Certainly, $\bar{\mu}_1$ is countably complete, so it suffices to show $\bar{\mu}_1\in D_X$.
            
            Let $f:\R_X\to \delta_X$ be a surjection definable in $D_X$ from $\Sigma_X$ and some $x\in\R_X$.\footnote{Any set in $D$ is ordinal definable from $\Sigma$ and some $x\in\R^*$. Then by elementarity, any set in $D_X$ is ordinal definable in $D_X$ from $\Sigma_X$ and some $x\in\R_X$. Minimizing the ordinal parameters in the definition, we can find such an $f$ definable in $D_X$ from $\Sigma_X$ and $x$.} Note $\pi_X(f):\R^*\to \delta$ is then definable in $D$ from $\Sigma$ and $x$. For $A\in P(\delta_X)^{D_X}$, let $G^f_A$ be the game (in $D_X$) from the proof of Theorem 28.15 of \cite{kanamori}. Similarly, for $B\in P(\delta)^D$, let $G^{\pi_X(f)}_B$ be the analogous game played in $D$, defined relative to the surjection $\pi_X(f): \R^* \to \delta$.
            
            If $y\in\R_X$ is a winning strategy for Player I (Player II) in $G^f_A$, then, by elementarity of $\pi_X$, $y$ is also a winning strategy for Player I (Player II) in $G^{\pi_X(f)}_{\pi_X(A)}$. And $\pi_X(A)$ is the unique subset of $P(\delta)$ in $D$ such that $y$ is a winning strategy in $G^{\pi_X(f)}_{\pi_X(A)}$.\footnote{Again see the proof of Theorem 28.15 of \cite{kanamori}.} So $\pi_X(A) = B$ if and only if there is $y\in \R_X$ such that $y$ is a winning strategy for $A$ in $G^f_A$ and also a winning strategy for $B$ in $G^{\pi_X(f)}_B$. Thus $\pi_X\upharpoonright P(\delta_X)^{D_X}$ is ordinal definable in $D$ from $\{\R_X,\pi_X(f), f\}$. Since $D \models V = L(Lp^\Sigma(\R^*))$, $f\in (Lp^\Sigma(\R_X)))^D$, so $f$ is ordinal definable in $D$ from parameters in $\{\R_X,\Sigma\} \cup \R_X$. Thus $\pi_X\upharpoonright P(\delta_X)^{D_X}$, and therefore $\bar{\mu}_1$, is ordinal definable in $D$ from parameters in $\{\R_X,\Sigma\} \cup \R_X$. Then by Lemma 5.24 of \cite{sihmor}, $\bar{\mu}_1 \in C_\Gamma(\R_X)$.  Then super-small $\Gamma$-$\Sigma^*$ mouse capturing gives $\bar{\mu}_1\in Lp^\Sigma(\R_X)$. Then by Lemma \ref{fullness}, $\bar{\mu}_1\in D_X$.
        \end{proof}

        The rest of the proof of the claim is just as in the proof of Claim \ref{PII wins}.
    \end{proof}

    The claim gives a contradiction just as in the proof of Theorem \ref{weakly cpt thm}.
\end{proof}

\begin{remark}
\label{proving super-small mc}
    Assumption (\ref{super-small msc}) in Theorem \ref{cof at least kappa} is not necessary below superstrongs. For let $U$ be the Martin measure on Turing degrees (in $D$). Let $T^*$ be the tree in $Ult(D,U)$ represented by $[x\mapsto T]_U$. Then $D = L(T^*,\R^*)$.\footnote{$\supseteq$ is clear. $\subseteq$ is by Claim 4 of Theorem 2.3 of \cite{consistency_strength_adr}, plus that $D\models V = L(P(\R))$.} Then by assumption (\ref{smsc}), $L(T^*,\R^*)\models V = L(Lp^\Sigma(\R^*))$. Then for $U$-measure-one many $x\in\R^*$, $L(T,\sigma_x) \models V = L(Lp^\Sigma(\sigma_x))$ (where $\sigma_x = \{y\in\R^*: y<_T x\}$). Fix $y_0\in \R^*$ such that $x\geq_T y_0 \implies L(T,\sigma_x) \models V = L(Lp^\Sigma(\sigma_x))$.

    Let $X$ be as in the proof of Theorem \ref{cof at least kappa}. We may assume $y_0\in \R_X$. Let $x\in \R^*$ be Sacks generic over $D_X$.\footnote{We can find such an $x$ in $\R^*$ since $D_X$ is countable in $V[G]$.} Then $\R_X = \sigma_x$ and $L(T,\R_X) \models V = L(Lp^\Sigma(\R_X))$. It follows that $L(T,\R_X) \cap P(\R_X) \subseteq (Lp^\Sigma(\R_X))^D$.\footnote{Why? We need to show if $N \triangleleft (Lp^\Sigma(\R_X))^{L(T,\R_X)}$, then $N$ is countably iterable in $D$, so that $N \triangleleft (Lp^\Sigma(\R_X))^D$. But by elementarity of the ultrapower map $j_U: L(T,\R_X)\to L(T^*,\R^*)$, $j_U(N)$ is countably iterable in $D= L(T^*,\R^*)$. Since $N$ embeds into $j_U(N)$, $N$ is also countably iterable in $D$.} $C_\Gamma(\R_X) = L(T,\R_X)$ by Theorem 3.4 of \cite{steel2016}, so $C_\Gamma(\R_X) \subseteq Lp^\Sigma(\R_X)$. This is our use of assumption (2) in Subclaim \ref{measure is valid play}.
\end{remark}

\begin{theorem}
\label{derived model not Lp^Sigma}
    Assume $PFA + \kappa$ is a regular limit of Woodin cardinals. Then, for any hod pair $(P,\Sigma)$ in $D(V,\kappa)$, $D(V,\kappa)$ does not satisfy $V = L(Lp^\Sigma(\R^*))$ $+$ ``super-small $\Gamma$-$\Sigma^*$-mouse capturing holds on a cone.''
\end{theorem}

\begin{remark}
    If $D(V,\kappa)$ is sufficiently ``small'' and does not satisfy $AD_\R$, then $D(V,\kappa)$ satisfies $V = L(Lp^\Sigma(\R^*))$ $+$ ``super-small $\Gamma$-$\Sigma^*$-mouse capturing holds on a cone.''. So the theorem implies either the derived model is large or it satisfies $AD_\R$.
\end{remark}

\begin{proof}[Proof of Theorem \ref{derived model not Lp^Sigma}]
    Suppose not. Let $D = D(V,\kappa)$ and $\Theta = \Theta^D$. We have $D=L(Lp^\Sigma(\R^*))$ for some hod pair $(P,\Sigma)\in D$. As in \cite{derivedmodelpfa}, we may pick $(P,\Sigma)$ such that $P\in V$ and there is a symmetric term for $\Sigma$ in $V$.

    By the main result of \cite{coherent}, there is a club $S\subseteq \Theta$ and $\vec{C} = \langle C_\alpha:\alpha\in S\rangle$ such that
    \begin{enumerate}
        \item $C_\alpha \subseteq S$ and $C_\alpha$ is a closed subset of limit ordinals contained in $\alpha$.
        \item $cof(\alpha) >\omega \implies C_\alpha$ is unbounded in $\alpha$.
        \item $\beta \in C_\alpha \implies C_\beta = C_\alpha \cap \beta$.
    \end{enumerate}
    The proof of \cite{coherent} gives more. First, $\vec{C}$ is ordinal definable in $D$ from $\Sigma$, so $\vec{C}\in V$. Second, the construction assigns to each $\alpha \in \vec{C}$ some $\gamma_\alpha < \Theta$ such that for any $\beta\in C_\alpha$, there is a (fine-structural) embedding $\sigma_{\beta,\alpha}:D|\gamma_\beta\to D|\gamma_\alpha$. If additionally $\eta\in C_\beta$, then $\sigma_{\eta,\alpha} = \sigma_{\beta,\alpha}\circ \sigma_{\eta,\beta}$.

    As we are assuming $D\models V = L(Lp^\Sigma(\R))$, \cite{derivedmodelpfa} gives $\Theta < \kappa^+$. Then by Theorem \ref{cof at least kappa}, $cof(\Theta) = \kappa$. Let $f:\kappa\to S$ be strictly increasing, continuous, and cofinal in $\Theta$. For limit ordinals $\alpha < \kappa$, let $C'_\alpha = f^{-1}(C_{f(\alpha)})$. Consider $\alpha< \kappa$ of uncountable cofinality (in $V$). Since $f$ is continuous, $f''\alpha$ is a club in $f(\alpha)$, so $f''\alpha \cap C_{f(\alpha)}$ is unbounded in $f(\alpha)$. Thus $C'_\alpha$ is a club in $\alpha$. For $\alpha<\kappa$ with $cof(\alpha)=\omega$, if $C'_\alpha$ is not cofinal in $\alpha$, replace $C'_\alpha$ by some sequence of order-type $\omega$ cofinal in $\alpha$. We now have $C'_\alpha$ is a club in $\alpha$ for limit $\alpha <\kappa$. 
        
    Suppose $\beta$ is a limit point of $C'_\alpha$. We must be in the case $C'_\alpha = f^{-1}(C_{f(\alpha)})$. Then $f(\beta)$ is a limit point of $C_{f(\alpha)} \cap f''\beta$, so $C_{f(\beta)} = C_{f(\alpha)} \cap f(\beta)$. Then $f''\beta \cap C_{f(\beta)} = f''\beta \cap C_{f(\alpha)}$ is cofinal in $f(\beta)$, so $C'_\beta = f^{-1}(C_{f(\beta)})$. Thus $C'_\beta = C'_\alpha \cap \beta$.
        
    By Theorem 1 of \cite{note_on_pfa}, there is a club $C' \subset \kappa$ such that if $\alpha$ is a limit point of $C'$, then $C'\cap \alpha = C'_\alpha$. Let $C = f''[C']$. Clearly, $C$ is a club in $\Theta$ and $C\subseteq S$. Suppose $\alpha$ and $\beta$ are limit points of $C$, and $\beta < \alpha$. There are $\beta',\alpha'<\kappa$ limit points of $C'$ such that $\beta = f(\beta')$ and $\alpha = f(\alpha')$. Then $C'_\alpha \cap \beta = (C'\cap \alpha) \cap \beta = C'\cap \beta = C'_\beta$. Then $C_\alpha \supseteq f''C'_{\alpha'} \supseteq f''C'_{\beta'}$, so $C_\alpha$ is cofinal in $\beta$. So $\beta$ is a limit point of $C_\alpha$. In particular, $\beta\in C_\alpha$.

    We have constructed a club $C\subset \Theta$ such that for $\eta < \beta < \alpha$ limit points of $C$, $\sigma_{\eta,\alpha} = \sigma_{\beta,\alpha}\circ \sigma_{\eta,\beta}$. Let $N$ be the direct limit of the system of embeddings $\langle \sigma_{\beta,\alpha}: \beta< \alpha \text{ limit points of } C \rangle$. For $\beta$ a limit point of $C$, let $\sigma_\beta:D|\gamma_\beta \to N$ be the direct limit embedding. $N$ is an $\Sigma$-premouse over $\R^*$ which projects to $\R^*$.\footnote{This follows from that $\sigma_\beta$ is sufficiently elementary, which is immediate from that the embeddings $\sigma_{\beta,\alpha}$ are sufficiently elementary, which is also shown during the construction of the coherent sequence in \cite{coherent}.} $N$ is countably iterable in $V[G]$, since any countable $R\prec N$ is contained in $range(\sigma_\beta)$ for large enough $\beta$ and $D|\gamma_\beta$ is countably iterable in $V[G]$.

    We have shown $N\triangleleft (Lp^\Sigma(\R^*))^{V[G]}$. Since $N\cap On > \Theta$, $N\notin D$. This gives a contradiction exactly as in the proof of Claim \ref{scales analysis}.
\end{proof}

\printbibliography

\end{document}